\documentclass[a4paper,11pt]{amsart}
\usepackage{amsmath,amsthm,amssymb,enumitem}
\usepackage[hidelinks]{hyperref}
\usepackage[T1]{fontenc}
\usepackage{lmodern}

\newtheorem{theorem}{Theorem}[section]
\newtheorem{proposition}[theorem]{Proposition}
\newtheorem{lemma}[theorem]{Lemma}
\newtheorem{corollary}[theorem]{Corollary}

\theoremstyle{definition}
\newtheorem{definition}[theorem]{Definition}
\newtheorem{problem}[theorem]{Problem}

\numberwithin{equation}{section}
\setlist[enumerate,1]{label=\upshape(\arabic*)}

\newcommand{\id}{\mathrm{id}}
\newcommand{\D}{\mathbb D}
\newcommand{\C}{\mathbb C}
\newcommand{\B}{\mathbb B}
\newcommand{\Pj}{\mathbb P}
\newcommand{\cO}{\mathcal O}
\newcommand{\Spl}{\mathrm{Spl}}
\DeclareMathOperator{\rank}{rank}
\DeclareMathOperator{\dist}{dist}
\DeclareMathOperator{\Hom}{Hom}

\begin{document}

\title[Every elliptic domain is meromorphically pseudoconvex]{Every Gromov elliptic domain is meromorphically pseudoconvex}
\author{Yuta Kusakabe}
\address{Faculty of Mathematics, Kyushu University, Fukuoka 819-0395, Japan}
\email{kusakabe@math.kyushu-u.ac.jp}

\begin{abstract}
We prove that every Gromov elliptic domain $\Omega$ in a complex manifold has the following property, which we call meromorphic pseudoconvexity: each boundary point has a neighborhood $V$ such that $\Omega\cap V$ becomes Stein after removing a set which is, locally in $\Omega\cap V$, the union of the polar sets of finitely many meromorphic functions on $\Omega\cap V$.
Since meromorphic functions on Hartogs figures extend to their hulls, it follows that, for a domain $\Omega=\operatorname{int}\overline\Omega$ in a Stein manifold, Gromov ellipticity is equivalent to being Oka and Stein, and also to being Oka and admitting a complete K\"ahler metric.
This complements a theorem of Matsumura and the author that every weakly pseudoconvex Oka manifold admitting a positive line bundle is elliptic.
As applications, we show that the Oka manifolds $\C^n\setminus\overline{\B^n}$ and $\C^n\setminus\mathbb R^n$, $n\geq2$, are not elliptic, and that no proper domain with Lipschitz boundary in $\Pj^n$, $n\geq2$, is elliptic.
\end{abstract}

\maketitle

\section{Introduction}

Stein manifolds are the natural domains of holomorphic functions, and the solution of the Levi problem characterizes Stein domains in a Stein manifold by pseudoconvexity at the boundary \cite{Docquier1960MA,Oka1953JJM}.
Dually, Oka manifolds are the natural targets of holomorphic maps \cite{Forstneric2017}, and the \emph{dual Levi problem} asks which geometric conditions on the boundary of a domain make it Oka (cf.\ \cite[Problem~3.10]{Forstneric2026ICM}, where part (a) is called the inverse Levi problem).
The known sufficient conditions are of a pseudoconcave nature: complements of compact holomorphically convex sets in $\C^n$, $n\geq2$, are Oka \cite{Kusakabe2024AoM}, and so are pseudoconcave tubes in certain line bundles \cite{Forstneric2025MA}.
In the other direction, an Oka manifold carries no nonconstant plurisubharmonic function bounded above \cite[Proposition~7.1.9]{Forstneric2017}, whereas a domain with a smooth strictly pseudoconvex boundary point carries such a function, so an Oka domain has no such boundary points (cf.\ \cite[Example~5]{Forstneric2014MZ}).

A fundamental sufficient condition for the Oka property is Gromov ellipticity \cite{Gromov1989JAMS}.
A complex manifold $Y$ is \emph{Gromov elliptic}, or simply \emph{elliptic}, if it admits a \emph{dominating spray}: a holomorphic vector bundle $p:E\to Y$ with a holomorphic map $s:E\to Y$ such that $s(0_y)=y$ and $ds_{0_y}(E_y)=T_{Y,y}$ for every $y\in Y$, where the fiber $E_y$ is identified with its tangent space at $0_y$.
Gromov asked whether every Oka manifold is elliptic \cite[3.2.A$''$]{Gromov1989JAMS}; the first nonelliptic Oka manifolds were found in \cite{Kusakabe2020PAMS}.
Since ellipticity is stronger than the Oka property, one might expect elliptic domains to be even more pseudoconcave.
We show that the opposite is true: what ellipticity adds is pseudoconvexity, up to the poles of meromorphic functions.
Some exceptional set is unavoidable, since the complement of a nonempty algebraic subset of codimension at least two in $\C^n$ is elliptic \cite{Gromov1989JAMS} but not pseudoconvex.

\begin{definition}
	\label{definition:meromorphic}
	Let $\Omega$ be a domain in a complex manifold $Y$.
	Here and throughout, the \emph{polar set} of a meromorphic function is the set of points where it is not holomorphic, including points of indeterminacy.
	We say that $\Omega$ is \emph{meromorphically pseudoconvex in} an open set $V\subset Y$ if $(\Omega\cap V)\setminus P$ is Stein for some set $P\subset\Omega\cap V$ which is, locally in $\Omega\cap V$, the union of the polar sets of finitely many meromorphic functions on $\Omega\cap V$.
	We say that $\Omega$ is \emph{meromorphically pseudoconvex} if every point of $\partial\Omega$ has an open neighborhood in which $\Omega$ is meromorphically pseudoconvex.
\end{definition}

Meromorphic pseudoconvexity should not be confused with \emph{meromorphic convexity}, which concerns hulls with respect to meromorphic functions \cite{Coltoiu1999,Boudreaux-Shafikov}.
For open subsets of Stein manifolds, meromorphic convexity with respect to their own meromorphic functions is equivalent to the Stein property \cite[Proposition~6.10]{Biller2007}, as is being a domain of meromorphy \cite{Kajiwara-Sakai1967}.

The set $P$ in Definition~\ref{definition:meromorphic} is a closed analytic subset of $\Omega\cap V$ of pure codimension one.
Taking $P=\emptyset$ near every boundary point recovers the usual notion of \emph{pseudoconvexity}.
Our main result is the following.

\begin{theorem}
	\label{theorem:main}
	Every Gromov elliptic domain in a complex manifold is meromorphically pseudoconvex.
\end{theorem}

For example, $\C^2\setminus\{0\}$ is elliptic, and it becomes pseudoconvex after the polar set $\{z_1=0\}\setminus\{0\}$ of the meromorphic function $1/z_1$ is removed.

By Levi's extension theorem, every meromorphic function on a Hartogs figure extends meromorphically to its hull \cite{Levi1910} (see also \cite{Siu1974}).
Consequently, if the image of a Hartogs figure $\Phi$ near a boundary point of a meromorphically pseudoconvex domain lies in the domain, then, for every $0<\rho<1$, the part of its smaller hull $\Phi(\rho\D^n)$ outside the domain lies in a hypersurface of $\Phi(\rho\D^n)$ (Proposition~\ref{proposition:hull}).
For domains that coincide with the interior of their closure, this forces pseudoconvexity (Proposition~\ref{proposition:local}), so Theorem~\ref{theorem:main} has the following consequence.

\begin{corollary}
	\label{corollary:equivalence}
	Let $\Omega$ be a domain in a Stein manifold $Y$ with $\Omega=\operatorname{int}\overline\Omega$.
	The following conditions are equivalent:
	\begin{enumerate}
		\item $\Omega$ is Gromov elliptic;
		\item $\Omega$ is Oka and Stein;
		\item $\Omega$ is Oka and admits a complete K\"ahler metric.
	\end{enumerate}
\end{corollary}

The condition $\Omega=\operatorname{int}\overline\Omega$ cannot be dropped, as $\C^2\setminus\{0\}$ shows, and it holds, for example, if $\Omega$ has Lipschitz boundary, that is, if near each boundary point, $\Omega$ is the region above the graph of a Lipschitz function in suitable local real coordinates.
The equivalence with (3) relies on the theorem of Diederich and Pflug that a domain $G=\operatorname{int}\overline G$ in a Stein manifold is Stein if it admits a complete K\"ahler metric \cite[Theorem~1(b)]{Diederich1981MA}.
Their hypothesis $G=\operatorname{int}\overline G$ is needed for the same reason as ours: complements of analytic subsets of Stein manifolds admit complete K\"ahler metrics \cite{Grauert1956MA}.

\begin{problem}
	\label{problem:Kahler}
	Does every Gromov elliptic domain in a Stein manifold admit a complete K\"ahler metric?
\end{problem}

In joint work with Matsumura \cite{Kusakabe-Matsumura}, we proved that every weakly pseudoconvex (weakly $1$-complete) Oka manifold admitting a positive line bundle is Gromov elliptic, generalizing the cases of Stein manifolds \cite{Gromov1989JAMS} and projective manifolds \cite{Forstneric2026MRL}.
Thus, roughly,
\begin{align*}
	\text{Gromov elliptic}&\approx\text{Oka}+\text{pseudoconvex},\quad\text{just as}\\
	\text{complete Kobayashi hyperbolic}&\approx\text{Kobayashi hyperbolic}+\text{pseudoconvex}.
\end{align*}
Every bounded domain is Kobayashi hyperbolic, so hyperbolicity alone does not imply pseudoconvexity, whereas complete hyperbolic domains are taut, hence pseudoconvex; conversely, many bounded pseudoconvex domains, including strictly pseudoconvex and convex ones, are complete hyperbolic \cite{Kobayashi1998}.

Proposition~\ref{proposition:local} gives further obstructions to ellipticity, which yield the following examples.

\begin{corollary}
	\label{corollary:applications}
	Let $n\geq2$.
	\begin{enumerate}
		\item If $K\subset\C^n$ is a nonempty compact set with $K=\overline{\operatorname{int}K}$ and connected complement, then $\C^n\setminus K$ is not Gromov elliptic.
		In particular, ${\C^n\setminus\overline{\B^n}}$ is an Oka manifold which is not Gromov elliptic.
		\item The complement $\C^n\setminus\mathbb R^n$ is an Oka manifold which is not Gromov elliptic.
		\item No domain $\Omega\subsetneq\Pj^n$ with Lipschitz boundary is Gromov elliptic.
	\end{enumerate}
\end{corollary}

For $n\geq3$, the nonellipticity in (1) and (2) also follows from results of Andrist, Shcherbina and Wold \cite[Theorem~1.1 and Corollary~1.4]{Andrist2016AM}; their argument extends holomorphic vector bundles across compact sets and closed totally real Lipschitz sets, which requires $n\geq3$.
For (2), recall that $\C^n\setminus\mathbb R^n$ is Oka by \cite[Corollary~1.7]{Kusakabe2024AoM} if $n\geq4$, by \cite[Theorem~1.3]{Du2026} if $n=3$, and by \cite[Theorem~1.2(i)]{Du-Guo-Wang-Xie2026} if $n=2$.

We outline the proof of Theorem~\ref{theorem:main}.
Fix a coordinate ball $U$.
Over small open subsets of $\Omega\cap U$, a holomorphic right inverse $\sigma$ of $ds|_E$ gives maps
\[
	\xi\mapsto s(\sigma_y\xi)=y+\xi+\sum_{|\alpha|\geq2}A_\alpha(y)\xi^\alpha,
\]
and the growth of their Taylor coefficients yields a plurisubharmonic function $u$ with $\dist(\cdot,U\setminus\Omega)\gtrsim e^{-u}$ (Lemma~\ref{lemma:covering}).
Since such right inverses are neither unique nor globally defined, we select finitely many of these maps by suitable affine conditions on their Taylor coefficients, independently of the choice of $\sigma$ (Section~\ref{section:canonical}); this is inspired by the use of Taylor coefficients in \cite[Examples~2.9 and~2.10]{Kusakabe-Matsumura}.
The symmetric functions of their coefficients are meromorphic on $\Omega\cap U$, and the maximum of the resulting functions $u$ extends plurisubharmonically outside their polar sets (Section~\ref{section:barrier}).
The corollaries are then deduced from Theorem~\ref{theorem:main} by Levi's extension theorem (Sections~\ref{section:main} and~\ref{section:applications}).

\section{Preliminaries}
\label{section:preliminaries}

Throughout, complex manifolds are paracompact and $n=\dim Y$.
A \emph{coordinate ball} in $Y$ is an open set $U\subset Y$ together with a biholomorphic map onto the unit ball $\B^n\subset\C^n$, by which we identify $U$ with $\B^n$; balls in $U$ and distances in $U$ are taken in these coordinates.
We use the convention $\dist(y,\emptyset)=+\infty$.
A \emph{hypersurface} of a complex manifold $N$ is a closed analytic subset of $N$ of pure codimension one; it may be empty.
Recall that an open set $G\subset\C^n$ is pseudoconvex if and only if $-\log\dist(\cdot,\C^n\setminus G)$ is plurisubharmonic on $G$, and if and only if $G$ admits a continuous plurisubharmonic exhaustion function \cite[Theorem~II.3.4 and Lemma~II.2.10]{Fritzsche2002}; pseudoconvex open subsets of $\C^n$ are Stein by Oka's solution of the Levi problem \cite{Oka1953JJM}.
Since every open subset of $\C$ is pseudoconvex, we may and do assume $n\geq2$ whenever pseudoconvexity is at stake.

Until the end of Section~\ref{section:barrier}, $\Omega\subset Y$ is a domain with a dominating spray $(E,p,s)$.
The differentials $ds_{0_y}|_{E_y}:E_y\to T_{Y,y}$ form a surjective morphism of holomorphic vector bundles $ds|_E:E\to T_\Omega$.
Its kernel $\mathcal K$ is a holomorphic subbundle of $E$ of rank $k$.

\subsection{A criterion for pseudoconvexity}
Let $U$ be a coordinate ball, and put $\delta_U(y)=1-|y|$ for $y\in U$.
For a function $u$ on an open set $V\subset\Omega\cap U$, we consider the inequality
\begin{equation}
	\label{equation:barrier}
	\dist(y,U\setminus\Omega)\geq c_ne^{-u(y)}\qquad(y\in V),
\end{equation}
where $c_n>0$ is the constant of Lemma~\ref{lemma:covering} below, which depends only on $n$.

\begin{lemma}
	\label{lemma:barrier}
	Let $P$ be a hypersurface of $\Omega\cap U$, and let $u$ be a plurisubharmonic function on $(\Omega\cap U)\setminus P$ satisfying \eqref{equation:barrier}.
	Then $(\Omega\cap U)\setminus P$ is pseudoconvex.
\end{lemma}

\begin{proof}
	Put $G=(\Omega\cap U)\setminus P$ and $G_t=\{y\in G:\max\{u(y),-\log\delta_U(y)\}<t\}$ for $t\in\mathbb R$.
	By \eqref{equation:barrier}, $\overline{G_t}$ is a compact subset of $\Omega\cap U$.
	Every point of $\overline{G_t}$ has a ball neighborhood $B\subset\Omega\cap U$, and $G_t\cap B=\{x\in B\setminus P:\varphi(x)<t\}$, where $\varphi=\max\{u,-\log\delta_U\}$ is plurisubharmonic on $B\setminus P$.
	The open set $B\setminus P$ is pseudoconvex, and so is $G_t\cap B$: if $\Phi$ is a Hartogs figure with $\Phi(H)\subset G_t\cap B$ (see Section~\ref{section:main}), then $\Phi(\D^n)\subset B\setminus P$, and for $(z_0,w_0)\in\D^n$ and $\max\{a,|z_0|\}<r<1$, the subharmonic function $z\mapsto\varphi(\Phi(z,w_0))$ is less than $t$ on the circle $|z|=r$, hence at $z_0$.
	Since pseudoconvexity is a local property of the boundary \cite[Theorem~II.3.7]{Fritzsche2002}, $G_t$ is pseudoconvex.
	Since $G=\bigcup_{j\geq1}G_j$ is an increasing union of pseudoconvex open sets, it is pseudoconvex \cite[Theorem~II.3.6]{Fritzsche2002}.
\end{proof}

\subsection{Taylor coefficients and the boundary distance}
A \emph{linear splitting at $y\in\Omega\cap U$} is a linear map $\sigma:\C^n\to E_y$ with $ds_{0_y}\circ\sigma=\id$, where $T_{Y,y}$ is identified with $\C^n$ by the coordinates of $U$.
For a holomorphic family $\sigma=(\sigma_y)_{y\in V}$ of linear splittings over an open set $V\subset\Omega\cap U$, write
\[
	s(\sigma_y\xi)=y+\xi+\sum_{|\alpha|\geq2}A_\alpha(y)\xi^\alpha
\]
near $\xi=0$ in the coordinates of $U$, where $\alpha$ runs over multi-indices and $A_\alpha(y)\in\C^n$, and define
\[
	v_\sigma=\biggl(\sup_{|\alpha|\geq2,\ 1\leq l\leq n}\frac{\log^+|(A_\alpha)_l|}{|\alpha|-1}\biggr)^*,
	\qquad
	u_\sigma=\max\{v_\sigma,-\log\delta_U\},
\]
where $\log^+t=\max\{0,\log t\}$ with $\log^+0=0$, and ${}^*$ denotes the upper semicontinuous regularization.

\begin{lemma}
	\label{lemma:covering}
	The functions $v_\sigma$ and $u_\sigma$ are plurisubharmonic on $V$, and there is a constant $c_n>0$ depending only on $n$ such that
	\[
		\dist(y,U\setminus\Omega)\geq c_ne^{-u_\sigma(y)}\qquad(y\in V).
	\]
\end{lemma}

\begin{proof}
	Locally in $y$, the map $(y,\xi)\mapsto s(\sigma_y\xi)$ takes values in $U$ for $\xi$ in a fixed small polydisc, so Cauchy's estimates give $|(A_\alpha(y))_l|\leq C\varrho^{-|\alpha|}$ with $C,\varrho>0$ independent of $y$ and $\alpha$.
	Hence the plurisubharmonic functions in the supremum are locally uniformly bounded above, and $v_\sigma$ is plurisubharmonic and finite.
	So is $u_\sigma$, since $-\log\delta_U$ is plurisubharmonic on the ball.

	Fix $y\in V$ and put $M=e^{v_\sigma(y)}\geq1$, so that $|(A_\alpha(y))_l|\leq M^{|\alpha|-1}$.
	Choose $\epsilon_n\in(0,1/4)$ such that
	\[
		n\sum_{m\geq2}m\binom{n+m-1}{m}\epsilon_n^{m-1}\leq\frac14,
	\]
	and put $r=\epsilon_n\min\{M^{-1},\delta_U(y)\}=\epsilon_ne^{-u_\sigma(y)}$.
	Since there are $\binom{n+m-1}{m}$ multi-indices of length $m$ and $Mr\leq\epsilon_n$, the series
	\[
		R_y(\xi)=\sum_{|\alpha|\geq2}A_\alpha(y)\xi^\alpha
	\]
	defines a holomorphic map near $r\overline{\B^n}$ with $R_y(0)=0$ and $\|dR_y\|\leq1/4$ on $r\overline{\B^n}$, where $\|\cdot\|$ is the operator norm.
	Put $F(\xi)=y+\xi+R_y(\xi)$.
	In particular, $|F(\xi)-y|\leq5r/4<\delta_U(y)$ there, so $F$ takes values in $U$ and agrees with $s(\sigma_y\cdot)$ on $r\overline{\B^n}$ by the identity theorem.
	For $|\eta|<r/2$, the map $\xi\mapsto\eta-R_y(\xi)$ is a contraction of $r\overline{\B^n}$ into itself, and its fixed point $\xi$ satisfies $y+\eta=F(\xi)=s(\sigma_y\xi)\in\Omega$.
	Hence $y+(r/2)\B^n\subset\Omega\cap U$, which proves the inequality with $c_n=\epsilon_n/2$.
\end{proof}

\section{Canonical restricted sprays}
\label{section:canonical}

Let $U$ be a coordinate ball and $D$ a connected component of $\Omega\cap U$.
For $y\in D$, the set $\Spl_y$ of linear splittings at $y$ is an affine space over $\Hom(\C^n,\mathcal K_y)\cong\C^{nk}$, and these spaces form a holomorphic affine bundle $\Spl\to D$ with connected total space.
For $\sigma\in\Spl_y$, the \emph{restricted spray} $S_\sigma:\C^n\to\Omega$, $S_\sigma(\xi)=s(\sigma\xi)$, satisfies $S_\sigma(0)=y$ and $dS_\sigma(0)=\id$.
Near $\xi=0$ we expand it in the coordinates of $U$ as
\[
	S_\sigma(\xi)=y+\xi+\sum_{|\alpha|\geq2}A_\alpha(\sigma)\xi^\alpha,\qquad A_\alpha(\sigma)\in\C^n.
\]
The point is that these coefficients take values in a fixed vector space and do not involve the bundle $E$.
By the identity theorem, two restricted sprays based at the same point coincide if and only if all their Taylor coefficients agree.

\subsection{Local description}
On a sufficiently small ball $D_0\subset D$, choose a holomorphic splitting $\sigma^0$ of $ds|_E$ and a holomorphic frame $\kappa_1,\dots,\kappa_k$ of $\mathcal K$.
For $\tau=(\tau_{jl})\in\C^\nu$, $\nu=nk$, put
\[
	\sigma_\tau(\xi)=\sigma^0(\xi)+\sum_{l=1}^k\Bigl(\sum_{j=1}^n\tau_{jl}\xi_j\Bigr)\kappa_l,
\]
which identifies $\Spl_y$ with $\C^\nu$ for every $y\in D_0$.
In a frame of $E$ over $D_0$, the spray has an expansion $s(v)=y+\sum_{j\geq1}\eta_j(y)(v)$ near the zero section, where $\eta_j(y)$ is a $\C^n$-valued homogeneous polynomial of degree $j$ on $E_y$ whose coefficients are holomorphic in $y$ (cf.\ \cite[Examples~2.9 and~2.10]{Kusakabe-Matsumura}).
Substituting $v=\sigma_\tau(\xi)$, which is linear in $\xi$ and affine in $\tau$, we see that
\[
	A_\alpha(y,\tau):=A_\alpha(\sigma_\tau)
\]
is a polynomial in $\tau$ of degree at most $|\alpha|$ whose coefficients are holomorphic functions of $y\in D_0$.
Another choice of $\sigma^0$ and $\kappa$ changes $\tau$ by an affine automorphism of $\C^\nu$ depending holomorphically on $y$, but does not change the functions $A_\alpha$ on $\Spl_y$.
For $J\geq2$ let $V_J=(\C^n)^{\{\alpha:2\leq|\alpha|\leq J\}}$ and
\[
	\Theta_J:\Spl\to V_J,\qquad
	\Theta_J(\sigma)=\bigl(A_\alpha(\sigma)\bigr)_{2\leq|\alpha|\leq J}.
\]
In local coordinates we write $\Theta_J(y,\tau)=\Theta_J(\sigma_\tau)$.

\subsection{Generic algebraic fibers}
We use Chevalley's theorem \cite[Exercises~II.3.18 and~II.3.19]{Hartshorne1977}, generic smoothness \cite[Corollary~III.10.7]{Hartshorne1977}, and generic topological local triviality of algebraic families \cite[Corollaire~5.1]{Verdier1976IM}.

\begin{lemma}[Finite selection of restricted sprays]
	\label{lemma:canonical}
	There are integers $J\geq2$, $r\geq0$ and $c\geq1$, a coordinate projection $\Lambda:V_J\to\C^r$, a point $\zeta\in\C^r$, and a closed set $\Sigma\subset D$ such that every point of $D$ has a connected neighborhood $W$ with $\Sigma\cap W$ contained in a nowhere dense analytic subset of $W$.
	For $y\in D$ put
	\[
		\Gamma_y=\{\sigma\in\Spl_y:\Lambda((A_\alpha(\sigma))_{2\leq|\alpha|\leq J})=\zeta\}.
	\]
	For $y\in D\setminus\Sigma$, these fibers have the following properties:
	\begin{enumerate}
		\item $\Gamma_y$ is a smooth algebraic set with $c$ connected components, and every $A_\alpha$ is constant on each component;
		\item near $y$ there are holomorphic families of splittings $\sigma_1,\dots,\sigma_c$ whose values at each point $y'$ lie in pairwise distinct components of $\Gamma_{y'}$;
		\item for every $\alpha$, $1\leq l\leq n$ and $1\leq j\leq c$, the function
		\[
			e_j\bigl((A_\alpha(\sigma_1))_l,\dots,(A_\alpha(\sigma_c))_l\bigr),
		\]
		where $e_j$ is the $j$-th elementary symmetric polynomial, does not depend on the choice of $\sigma_1,\dots,\sigma_c$, is holomorphic on $D\setminus\Sigma$, and extends meromorphically to $D$.
	\end{enumerate}
\end{lemma}

\begin{proof}
	\emph{Step~1: Rank.}
	Let $r_J$ be the maximal vertical rank of $\Theta_J$ on $\Spl$.
	These integers are nondecreasing in $J$ and at most $\nu$, so we may fix $J$ such that $r_{J'}=r_J=:r$ for all $J'\geq J$.
	Then every finite family of coefficients $A_\alpha$ has vertical rank at most $r$.
	If $r=0$, every $A_\alpha$ is constant on each $\Spl_y$, and the lemma holds with $c=1$ and $\Sigma=\emptyset$, using local holomorphic splittings.
	We henceforth assume $r>0$.
	The locus where $\Theta_J$ has vertical rank less than $r$ is an intrinsic proper analytic subset of the connected manifold $\Spl$, defined locally by the vanishing of the $r\times r$ minors of its vertical Jacobian.
	It has empty interior, so $\Theta_J$ has vertical rank $r$ somewhere over every nonempty open subset of $D$.

	\emph{Step~2: Choice of $\Lambda$ and $\zeta$.}
	Fix a ball $D_0\subset D$ as above.
	Choose a coordinate projection $\Lambda:V_J\to\C^r$ such that $\Lambda\circ\Theta_J$ has vertical rank $r$ somewhere over $D_0$, and define $F_y=(\Lambda\circ\Theta_J)|_{\Spl_y}$ for every $y\in D$.
	Collecting the coefficients of the polynomial map $\Theta_J(y,\cdot)$, we obtain a holomorphic map $\chi:D_0\to\C^Q$ and a universal polynomial map $\Theta^u$ such that $\Theta_J(y,\tau)=\Theta^u(\chi(y),\tau)$.
	Let $\mathcal Z$ be the Zariski closure of $\chi(D_0)$.
	By the choice of $\Lambda$, the map $\Lambda\Theta^u(z,\cdot)$ is dominant for $z$ in a nonempty Zariski open subset of $\mathcal Z$.
	Thus $\Psi(z,\tau)=(z,\Lambda\Theta^u(z,\tau))$ is dominant.
	Let $\mathcal Z_{\mathrm{reg}}$ be the regular locus of $\mathcal Z$.
	Generic smoothness, applied to the restriction $\Psi:\mathcal Z_{\mathrm{reg}}\times\C^\nu\to\mathcal Z_{\mathrm{reg}}\times\C^r$, gives a nonempty Zariski open subset of $\mathcal Z_{\mathrm{reg}}\times\C^r$ over which $\Psi$ has no critical points; since $\Psi$ preserves the first component, this means that $\Lambda\Theta^u(z,\cdot)$ has rank $r$ at every point of the fiber.
	Dominance and Chevalley's theorem give a nonempty Zariski open subset of $\mathcal Z_{\mathrm{reg}}\times\C^r$ contained in the image of this restriction; intersecting these open sets gives a nonempty Zariski open subset $\mathcal V\subset\mathcal Z_{\mathrm{reg}}\times\C^r$ of pairs with nonempty fibers without critical points.
	The set of all such pairs $(z,\zeta)$ is constructible in $\mathcal Z\times\C^r$, by Chevalley's theorem applied on $\mathcal Z\times\C^\nu$ to $\Psi$ and to the locus where the vertical rank is less than $r$.
	Fix $\zeta$ such that the slice of $\mathcal V$ over $\zeta$ is nonempty; then $\Gamma_y=F_y^{-1}(\zeta)$ for $y\in D$.
	Call $y\in D$ \emph{good} if $\Gamma_y\neq\emptyset$ and $F_y$ has rank $r$ at every point of $\Gamma_y$.
	By the observation above, the points of $D_0$ which are not good lie in a nowhere dense analytic subset of $D_0$.

	\emph{Step~3: The exceptional set.}
	Let $\Gamma=\bigcup_{y\in D}\Gamma_y\subset\Spl$, and let $\Sigma$ be the set of points of $D$ having no open neighborhood $N$ consisting of good points such that $\Gamma|_N\to N$ is topologically locally trivial.
	Here local triviality means that each $y'\in N$ has an open neighborhood $W\subset N$ and a homeomorphism $\Gamma|_W\cong W\times\Gamma_{y'}$ commuting with the projections to $W$.
	Clearly $\Sigma$ is closed, and every point of $D\setminus\Sigma$ is good.
	Let $D_1\subset D$ be a ball as above which contains a nonempty open set of good points, and let $\chi_1$, $\mathcal Z_1$ and $\Theta_1^u$ be the corresponding coefficient map, Zariski closure and universal polynomial map.
	The parameters $z\in\mathcal Z_1$ whose fiber $\{\tau:\Lambda\Theta_1^u(z,\tau)=\zeta\}$ is empty or contains a critical point of $\Lambda\Theta_1^u(z,\cdot)$ form a constructible set.
	This set is not dense in $\mathcal Z_1$: otherwise it would contain a nonempty Zariski open subset of $\mathcal Z_1$, whose inverse image under $\chi_1$ is dense in $D_1$ by the observation above and consists of points which are not good.
	Removing its Zariski closure and applying generic topological local triviality, we obtain a nonempty Zariski open subset $\mathcal Z_1'\subset\mathcal Z_1$ over which the fibers are nonempty without critical points and the family $\{(z,\tau)\in\mathcal Z_1\times\C^\nu:\Lambda\Theta_1^u(z,\tau)=\zeta\}$ is topologically locally trivial.
	Pulling back by $\chi_1$, we see that $\Sigma\cap D_1$ is contained in $\chi_1^{-1}(\mathcal Z_1\setminus\mathcal Z_1')$, which is a nowhere dense analytic subset of $D_1$.
	Let $D^*$ be the set of points of $D$ having a ball neighborhood $B\subset D$ such that $\Sigma\cap B$ is contained in a nowhere dense analytic subset of $B$.
	This set is open and contains $D_0$, by the preceding argument with $D_1=D_0$.
	If $y\in\overline{D^*}\cap D$, choose a small ball $D_1\ni y$ as above; it meets a ball $B$ witnessing membership in $D^*$, so $(D_1\cap B)\setminus\Sigma$ is a nonempty open set of good points.
	The preceding argument applies to $D_1$, hence $y\in D^*$.
	Thus $D^*$ is also closed in $D$, and connectedness gives $D^*=D$.
	In particular, $D\setminus\Sigma$ is connected and dense in $D$, since a closed set locally contained in nowhere dense analytic subsets does not disconnect a connected complex manifold.

	\emph{Step~4: Components.}
	For $y\in D\setminus\Sigma$, the fiber $\Gamma_y$ is a smooth algebraic set of dimension $\nu-r$, so it has finitely many connected components.
	Their number is locally constant on the connected set $D\setminus\Sigma$, hence equal to an integer $c\geq1$.
	For every scalar coefficient $A=(A_\alpha)_l$, the stabilization of the rank gives
	\[
		\rank d(F_y,A)\leq r=\rank dF_y\quad\text{on }\Gamma_y,
		\qquad\text{hence}\qquad
		dA|_{T_{\Gamma_y}}=0.
	\]
	Thus every $A_\alpha$ is constant on each component of $\Gamma_y$.
	By the implicit function theorem, choose for each component of $\Gamma_y$ a local holomorphic section of $\Gamma\to D$ through a point of that component, and denote these sections by $\sigma_1,\dots,\sigma_c$.
	Let $W\subset D\setminus\Sigma$ be a connected neighborhood of $y$ on which they are defined and over which there is a trivialization $h:\Gamma|_W\cong W\times\Gamma_y$.
	For each $i$, the continuous map $y'\mapsto\mathrm{pr}_2(h(\sigma_i(y')))$ takes values in a single component of $\Gamma_y$, and these components are pairwise distinct, since $h$ restricts to a homeomorphism of $\Gamma_y$ and $\sigma_1(y),\dots,\sigma_c(y)$ lie in distinct components.
	As $h$ restricts to homeomorphisms $\Gamma_{y'}\cong\Gamma_y$, it follows that $\sigma_1(y'),\dots,\sigma_c(y')$ lie in pairwise distinct components of $\Gamma_{y'}$ for every $y'\in W$.
	This proves (1) and (2).
	Any two such choices give the same multiset of values of $A_\alpha$ at each point, since each choice meets every component exactly once and $A_\alpha$ is constant on components.
	Thus their elementary symmetric functions agree on overlaps and define holomorphic functions on $D\setminus\Sigma$, independently of the sections.

	\emph{Step~5: Meromorphic extension.}
	Fix a scalar coefficient $A=(A_\alpha)_l$ and a ball $B\subset D$ as in the local description, so small that $\Sigma\cap B$ is contained in a nowhere dense analytic subset $\Sigma'$ of $B$.
	Collect the coefficients of both $\Theta_J(y,\cdot)$ and $A(y,\cdot)$ in a holomorphic map $\chi_B:B\to\C^{Q_B}$, let $\mathcal Z_B$ be the Zariski closure of $\chi_B(B)$, and write $\Theta_J(y,\tau)=\Theta_B^u(\chi_B(y),\tau)$ and $A(y,\tau)=A_B^u(\chi_B(y),\tau)$.
	The image $I$ of the algebraic family $\{(z,\tau)\in\mathcal Z_B\times\C^\nu:\Lambda\Theta_B^u(z,\tau)=\zeta\}$ under $(z,\tau)\mapsto(z,A_B^u(z,\tau))$ is constructible, and its fibers over $\chi_B(B\setminus\Sigma)$ are finite by (1).
	Hence $I$ is not dense in $\mathcal Z_B\times\C$: otherwise it would contain a nonempty Zariski open subset, whose fibers are infinite over a nonempty Zariski open subset of $\mathcal Z_B$, which meets the Zariski dense set $\chi_B(B\setminus\Sigma)$.
	Composing with $\chi_B$ the coefficients of a polynomial which vanishes on $I$ but not on $\mathcal Z_B\times\C$, we obtain holomorphic functions $q_0,\dots,q_m$ on $B$, not all identically zero, such that for every $y\in B\setminus\Sigma$ and every value $\mu$ of $A$ on $\Gamma_y$,
	\[
		q_m(y)\mu^m+\dots+q_1(y)\mu+q_0(y)=0.
	\]
	Discarding identically vanishing leading coefficients, we may assume $q_m\not\equiv0$.
	Since $\Gamma_y\neq\emptyset$ for every $y\in B\setminus\Sigma$, we have $m\geq1$: otherwise $q_0$ would vanish on the dense set $B\setminus\Sigma$.
	For any local choice of sections, put $\mu_i=A(\sigma_i)$, $1\leq i\leq c$.
	Cauchy's bound for roots gives $|q_m\mu_i|\leq |q_m|+\max_{0\leq h<m}|q_h|$, so these products are locally uniformly bounded near $\Sigma\cap B$.
	By Step~4, $q_m^je_j(\mu_1,\dots,\mu_c)=e_j(q_m\mu_1,\dots,q_m\mu_c)$ is a well-defined holomorphic function on $B\setminus\Sigma$.
	It is locally bounded near $\Sigma$, so its restriction to $B\setminus\Sigma'$ extends holomorphically to $B$ by Riemann's extension theorem, and the extension agrees with it on $\Sigma'\setminus\Sigma$ by continuity.
	Dividing by $q_m^j$, we obtain meromorphic extensions of the symmetric functions to $B$, which agree on overlaps by the identity theorem.
	This proves (3).
\end{proof}

\section{Proof of Theorem~\ref{theorem:main}}
\label{section:barrier}

Let $U\subset Y$ be a coordinate ball.
We show that $\Omega$ is meromorphically pseudoconvex in $U$; this proves Theorem~\ref{theorem:main}, since every point of $Y$ has a coordinate ball as a neighborhood.

Let $D$ be a connected component of $\Omega\cap U$, with the data of Lemma~\ref{lemma:canonical}.
On $D\setminus\Sigma$, choose local families $\sigma_1,\dots,\sigma_c$ as in Lemma~\ref{lemma:canonical}\,(2), and put
\[
	v=\max_{1\leq i\leq c}v_{\sigma_i},\qquad u=\max\{v,-\log\delta_U\}.
\]
To see that these functions do not depend on the choice of sections, write
\[
	w_i=\sup_{|\alpha|\geq2,\ 1\leq l\leq n}
	\frac{\log^+|(A_\alpha(\sigma_i))_l|}{|\alpha|-1},
	\qquad \widehat v=\max_{1\leq i\leq c}w_i.
\]
By Lemma~\ref{lemma:canonical}\,(1) and~(2), $\widehat v$ is intrinsic: at each point, any choice gives the same coefficient values up to permutation.
Since a finite maximum commutes with upper semicontinuous regularization, $v=\max_i w_i^*=\widehat v^*$ is intrinsic, and so is $u$.
The functions $v$ and $u$ are finite and plurisubharmonic on $D\setminus\Sigma$, with $v\geq0$.
Moreover, $u$ satisfies \eqref{equation:barrier} by Lemma~\ref{lemma:covering} applied to $\sigma_1$, since $u\geq u_{\sigma_1}$.

Let $P_D$ be the union of the polar sets of the meromorphic functions in Lemma~\ref{lemma:canonical}\,(3).
These polar sets are contained in $\Sigma$.
Near any point of $D$, choose a small connected neighborhood $W$ such that $\Sigma\cap W$ is contained in a nowhere dense analytic subset $\Sigma'_W$ with finitely many irreducible components.
The intersection of each polar set with $W$ is of pure codimension one and is contained in $\Sigma'_W$, so it is a union of codimension-one irreducible components of $\Sigma'_W$.
For each such component occurring in $P_D\cap W$, choose one of the meromorphic functions whose polar set contains it.
Then $P_D\cap W$ is the union of the intersections of these finitely many polar sets with $W$, hence is closed and of pure codimension one.
Thus $P_D$ is a hypersurface of $D$, all these meromorphic functions are holomorphic on $D\setminus P_D$, and locally $P_D$ is the union of the polar sets of finitely many of them.

\begin{lemma}
	\label{lemma:invisible}
	The function $v$ extends plurisubharmonically to $D\setminus P_D$, and the corresponding extension of $u$ satisfies \eqref{equation:barrier} on $D\setminus P_D$.
\end{lemma}

\begin{proof}
	Fix $y_0\in\Sigma\setminus P_D$, and choose a connected neighborhood $N\subset D\setminus P_D$ of $y_0$ with a nowhere dense analytic subset $\Sigma'$ containing $\Sigma\cap N$.
	Since $\Sigma'$ is nowhere dense, there is a complex line $\ell$ through $y_0$ on which $y_0$ is an isolated point of $\Sigma'$.
	Choose a disc $\Delta_n\subset\ell$ centered at $y_0$ whose closure lies in $N$ and meets $\Sigma'$ only at $y_0$.
	In affine coordinates centered at $y_0$ in which $\ell$ is the last coordinate axis, choose, by compactness of $\partial\Delta_n$, a polydisc $\Delta'\subset\C^{n-1}$ centered at $0$ so small that the polydisc $\Delta=\Delta'\times\Delta_n$ has $\overline\Delta\subset N$ and its vertical boundary
	\[
		L=\overline{\Delta'}\times\partial\Delta_n
	\]
	disjoint from $\Sigma'$.
	Since $L$ is a compact subset of $D\setminus\Sigma$, let $M\geq0$ bound $v$ on $L$.
	For $|\alpha|=m\geq2$ and $1\leq l\leq n$, the values $\mu_i=(A_\alpha(\sigma_i))_l$ satisfy $|\mu_i|\leq e^{(m-1)M}$ on $L$.
	Let $e_j$ denote the meromorphic extensions of their elementary symmetric functions from Lemma~\ref{lemma:canonical}\,(3).
	They are holomorphic on $D\setminus P_D$, hence near $\overline\Delta$, and off $\Sigma$ they are the elementary symmetric functions of $\mu_1,\dots,\mu_c$.
	The maximum principle on the vertical discs gives
	\[
		|e_j|\leq\binom cje^{j(m-1)M}\quad\text{on }\Delta.
	\]
	Since the $\mu_i$ are the roots of $T^c-e_1T^{c-1}+\dots+(-1)^ce_c$, the classical bound for roots of polynomials gives
	\[
		|\mu_i|\leq2\max_j|e_j|^{1/j}\leq2c\,e^{(m-1)M}\quad\text{on }\Delta\setminus\Sigma.
	\]
	Indeed, put $R=\max_j|e_j|^{1/j}$; if $R=0$, all roots vanish, while for $R>0$ and $|T|>2R$,
	\[
		\sum_{j=1}^c|e_j||T|^{c-j}
		\leq |T|^c\sum_{j=1}^c(R/|T|)^j<|T|^c,
	\]
	so $T$ cannot be a root.
	Hence $v\leq M+\log(2c)$ on $\Delta\setminus\Sigma$.
	By the removable singularity theorem for plurisubharmonic functions \cite{Grauert1956MZ}, $v|_{\Delta\setminus\Sigma'}$ extends to a plurisubharmonic function $\tilde v$ on $\Delta$ with $\tilde v(y)=\limsup_{x\to y,\,x\notin\Sigma'}v(x)$ for $y\in\Sigma'$.
	It agrees with $v$ on $\Delta\setminus\Sigma$, since plurisubharmonic functions which agree almost everywhere agree everywhere; for the same reason, such extensions agree on overlaps.
	Moreover, $0\leq\tilde v\leq M+\log(2c)$, so $\tilde u=\max\{\tilde v,-\log\delta_U\}$ is finite and extends $u$.
	For $y\in\Sigma\cap\Delta$, choose $x_k\in\Delta\setminus\Sigma'$ tending to $y$ with $v(x_k)\to\tilde v(y)$.
	If $U\setminus\Omega\neq\emptyset$, passing to the limit in \eqref{equation:barrier} using the continuity of $\delta_U$ and $\dist(\cdot,U\setminus\Omega)$ proves the same inequality for $\tilde u(y)$; if $U\setminus\Omega=\emptyset$, it holds by our convention.
\end{proof}

\begin{proof}[Proof of Theorem~\ref{theorem:main}]
	Put $P=\bigcup_DP_D$.
	Extending the meromorphic functions of Lemma~\ref{lemma:canonical}\,(3) for each component $D$ by zero to the other components of $\Omega\cap U$, we see from the discussion above that $P$ is, locally in $\Omega\cap U$, the union of the polar sets of finitely many meromorphic functions on $\Omega\cap U$.
	The functions provided by Lemma~\ref{lemma:invisible} on the disjoint open sets $D\setminus P_D$ combine into a plurisubharmonic function on $(\Omega\cap U)\setminus P$ satisfying \eqref{equation:barrier}.
	Thus $(\Omega\cap U)\setminus P$ is pseudoconvex by Lemma~\ref{lemma:barrier}, hence Stein.
	Therefore $\Omega$ is meromorphically pseudoconvex in $U$.
\end{proof}

\section{Hartogs hulls and pseudoconvexity}
\label{section:main}

For $0<a,b<1$ we consider the polydisc $\D^n\subset\C\times\C^{n-1}$ with coordinates $(z,w)$ and the Hartogs figure
\[
	H(a,b)=\{(z,w)\in\D^n:|z|>a\ \text{or}\ |w|_\infty<b\},\qquad|w|_\infty=\max_j|w_j|.
\]
A \emph{Hartogs figure} in a complex manifold $Y$ of dimension $n\geq2$ consists of parameters $0<a,b<1$ and a holomorphic embedding $\Phi$ of a neighborhood of $\overline{\D^n}$ into $Y$; its \emph{image} is $\Phi(H(a,b))$ and its \emph{hull} is $\Phi(\D^n)$.
We write $H=H(a,b)$ when the parameters are clear, and $\rho\D^n$ for the polydisc of polyradius $(\rho,\dots,\rho)$.
For $n\geq2$, an open set $G\subset\C^n$ is pseudoconvex if and only if $\Phi(\D^n)\subset G$ for every Hartogs figure $\Phi$ in $\C^n$ with $\Phi(H)\subset G$ \cite[Theorem~II.3.4]{Fritzsche2002}.
This characterization applies componentwise, since $H(a,b)$ is connected.
To obtain our figures from embeddings defined only on $\D^n$, if $\Phi(x_0)\notin G$, choose $\max\{a,b,|x_0|_\infty\}<r<1$ and replace $\Phi(x)$ by $\Phi(rx)$ with parameters $a/r,b/r$; the new embedding is defined near $\overline{\D^n}$ and its hull still contains $\Phi(x_0)$.
If the inner polydisc has different coordinate radii, replace them by their minimum to obtain a contained figure of the form $H(a,b)$ with the same hull.
If $G\subset B$ for a ball $B$, the hull of every such figure lies in $B$, since $B$ is pseudoconvex.
If $G$ is not pseudoconvex, we may therefore choose such a figure with $\Phi(\overline{\D^n})\subset B$ and $\Phi(\D^n)\not\subset G$, by shrinking its hull slightly and rescaling its parameters.

We next establish a continuity principle for level sets.

\begin{lemma}
	\label{lemma:level-sets}
	Let $G\subset\C^m$, $m\geq2$, be a pseudoconvex open set, $Q\subset\C^m$ a bounded domain, $F$ a holomorphic function on a neighborhood of $\overline Q$ which is not constant on $Q$, and $\gamma:[0,1]\to\C$ a path.
	Put $A_t=\{x\in\overline Q:F(x)=\gamma(t)\}$.
	If $A_0\subset G$ and $A_t\cap\partial Q\subset G$ for all $t$, then $A_t\subset G$ for all $t$.
\end{lemma}

\begin{proof}
	Let $\psi$ be a continuous plurisubharmonic exhaustion function of $G$ such that $\psi-|x|^2$ is plurisubharmonic; for instance, add $|x|^2$ to any continuous plurisubharmonic exhaustion function.
	Then, for every nonconstant holomorphic disc $f$ in $G$, the subharmonic function $\psi\circ f$ is constant on no nonempty open set, since $\psi\circ f-|f|^2$ is subharmonic and $|f|^2$ is harmonic on no nonempty open set; hence $\psi\circ f$ has no local maximum by the maximum principle.
	As every point of an analytic subset of $G$ of pure positive dimension lies on a nonconstant holomorphic disc in it, $\psi$ has no local maximum on such a subset.
	The set $L=\bigcup_t(A_t\cap\partial Q)$ is compact in $G$; choose $M$ with $\psi\leq M$ on $L$, and put $K=\{\psi\leq M\}\Subset G$.
	The set $I=\{t\in[0,1]:A_t\subset G\}$ is open by compactness of $\overline Q$, and $0\in I$.
	For $t\in I$, every irreducible component of $A_t\cap Q$ has dimension $m-1>0$, so its closure meets $\partial Q$, as $Q$ contains no compact analytic sets of positive dimension.
	For each such component $Z$, its closure is compact and $\overline Z\setminus Z\subset A_t\cap\partial Q\subset L$.
	Since $\overline Z\subset A_t\subset G$, the function $\psi$ attains its maximum over $\overline Z$ on $\overline Z\setminus Z\subset L$, so $Z\subset K$; hence $A_t\subset K$.
	If $t_j\in I$ converges to $t$, every point of $A_t\cap Q$ is a limit of points of $A_{t_j}\subset K$, since $F$ is open on $Q$.
	Together with $A_t\cap\partial Q\subset L$, this gives $A_t\subset G$.
	Thus $I$ is closed as well as open, so $I=[0,1]$.
\end{proof}

The next lemma allows a hypersurface to be removed from a Hartogs figure.
It reflects the pseudoconcavity of the complement of a pseudoconvex domain and the plurisubharmonicity of maxima over analytic multifunctions; see S{\l}odkowski \cite{Slodkowski1981MA} for a general framework.

\begin{lemma}
	\label{lemma:hypersurface}
	Let $0<a,b<1$, let $h\not\equiv0$ be a holomorphic function on $\D^n$, $n\geq2$, and put $T=\{h=0\}$.
	If $G\subset\D^n$ is a pseudoconvex open set with $H(a,b)\setminus T\subset G$, then $\D^n\setminus T\subset G$.
\end{lemma}

\begin{proof}
	Fix $a<\rho_1<1$.
	For $w\in\D^{n-1}$ let $K_w=\{z\in\C:|z|\leq\rho_1,\ (z,w)\notin G\}$, which is compact, and put $\varphi(w)=\max_{z\in K_w}\log|h(z,w)|$, where $\varphi(w)=-\infty$ if $K_w=\emptyset$.
	Since $G$ is open, $\varphi$ is upper semicontinuous.
	If $|w|_\infty<b$, then $\D\times\{w\}\subset H(a,b)$, so $K_w\subset\{z:h(z,w)=0\}$ and $\varphi(w)=-\infty$.
	We allow a plurisubharmonic function to be identically $-\infty$.
	We show below that $\varphi$ is plurisubharmonic; since it equals $-\infty$ on the nonempty open set $b\D^{n-1}$ and $\D^{n-1}$ is connected, it then follows that $\varphi\equiv-\infty$ on $\D^{n-1}$.
	Thus $K_w\subset\{z:h(z,w)=0\}$ for all $w$, that is, $(\rho_1\overline\D\times\D^{n-1})\setminus T\subset G$, and the lemma follows since $\rho_1<1$ is arbitrary.

	By the sub-mean value characterization of plurisubharmonicity and the density of real parts of holomorphic polynomials in the space of continuous functions on a circle, it suffices to prove the following.
	Let $w_0\in\D^{n-1}$, $\theta\in\C^{n-1}$ and $\epsilon>0$ be such that $w_0+\lambda\theta\in\D^{n-1}$ for $|\lambda|\leq\epsilon$, and let $q$ be a holomorphic polynomial with $\varphi(w_0+\lambda\theta)<\operatorname{Re}q(\lambda)$ for $|\lambda|=\epsilon$.
	Then $\varphi(w_0)\leq\operatorname{Re}q(0)$.
	For this reduction, approximate the upper semicontinuous boundary values by continuous functions $g_k$ decreasing to them, and choose real trigonometric polynomials $p_k$ with $g_k<p_k<g_k+1/k$ on the circle.
	Each $p_k$ is the real part of a holomorphic polynomial there, so the asserted comparison bounds $\varphi(w_0)$ by the mean of $p_k$; taking limits gives the sub-mean value inequality.

	Suppose to the contrary that $|h(z_0,w_0)|>e^{\operatorname{Re}q(0)}$ for some $z_0\in K_{w_0}$.
	The open set $G'=\{(z,\lambda):(z,w_0+\lambda\theta)\in G\}$ is Stein, hence pseudoconvex: the graph of the affine map $(z,\lambda)\mapsto(z,w_0+\lambda\theta)$ over $G'$ is a closed complex submanifold of the Stein manifold $\C^2\times G$.
	The function $F(z,\lambda)=h(z,w_0+\lambda\theta)e^{-q(\lambda)}$ is holomorphic near the closure of the bidisc $Q=\rho_1\D\times\epsilon\D$.
	By the choice of $q$, we have $|F|<1$ on the compact set $K'=\{(z,\lambda):|\lambda|=\epsilon,\ z\in K_{w_0+\lambda\theta}\}$, while $|F(z_0,0)|>1$.
	Replacing $F(z,\lambda)$ by $F(z,\lambda)e^{\eta z}$ for a small $\eta>0$, which preserves these strict inequalities by compactness, we may assume that $F$ is not constant, since $F(z,\lambda)e^{\eta z}$ is constant for at most one $\eta$.
	Put $c_0=F(z_0,0)$, and apply Lemma~\ref{lemma:level-sets} to $G'$, $Q$, $F$ and the path $\gamma(t)=((1-t)R+t)c_0$, where $R>1$ is so large that $A_0=\emptyset$.
	If $(z,\lambda)\in A_t\cap\partial Q$, then $|F(z,\lambda)|\geq|c_0|>1$, so $h(z,w_0+\lambda\theta)\neq0$ and $(z,\lambda)\notin K'$.
	Hence $(z,w_0+\lambda\theta)\in H(a,b)\setminus T\subset G$ if $|z|=\rho_1$, and $(z,w_0+\lambda\theta)\in G$ if $|\lambda|=\epsilon$.
	Thus $A_t\cap\partial Q\subset G'$ for all $t$, and Lemma~\ref{lemma:level-sets} gives $(z_0,0)\in A_1\subset G'$, contrary to $z_0\in K_{w_0}$.
\end{proof}

We now show that, for every $0<\rho<1$, the part of a smaller Hartogs hull $\Phi(\rho\D^n)$ outside a meromorphically pseudoconvex domain lies in a hypersurface of $\Phi(\rho\D^n)$.

\begin{proposition}
	\label{proposition:hull}
	Let $\Omega$ be a domain in a complex manifold $Y$ which is meromorphically pseudoconvex in an open set $V\subset Y$, and let $\Phi$ be a Hartogs figure in $Y$ such that $\Phi(H)\subset\Omega$ and $\Phi(\overline{\D^n})\subset V$.
	Then for every $0<\rho<1$ there is a hypersurface $T$ of $\rho\D^n$ such that $\Phi(\rho\D^n\setminus T)\subset\Omega$.
\end{proposition}

For example, if $\Omega=\C^2\setminus\{0\}$ and $V=\C^2$, the hull $\Phi(\D^2)$ of a Hartogs figure with $\Phi(H)\subset\Omega$ may contain the origin, and then $0\in\Phi(T)$ whenever $0\in\Phi(\rho\D^2)$.

\begin{proof}[Proof of Proposition~\ref{proposition:hull}]
	Since $T$ may be restricted to smaller polydiscs, we may assume $\rho>\max\{a,b\}$.
	Let $P\subset\Omega\cap V$ be as in Definition~\ref{definition:meromorphic}.
	Choose $a<a'<\rho$ and $0<b'<b$, and consider the Hartogs figure
	\[
		H'=\{(z,w)\in\rho\D^n:|z|>a'\ \text{or}\ |w|_\infty<b'\}
	\]
	of $\rho\D^n$, which is connected and has compact closure in $H$.
	Since $\Phi(\overline{H'})$ is a compact subset of $\Omega\cap V$, there are meromorphic functions $f_1,\dots,f_m$ on $\Omega\cap V$ such that $P\cap\Phi(\overline{H'})$ is contained in the union of their polar sets.
	By Levi's extension theorem for meromorphic functions on Hartogs figures \cite{Levi1910} (see also \cite{Siu1974}), each $f_i\circ\Phi|_{H'}$ extends to a meromorphic function on $\rho\D^n$.
	Denote these extensions by $\tilde f_1,\dots,\tilde f_m$.
	The germs $g$ of holomorphic functions with $g\tilde f_i$ holomorphic for all $i$ form a coherent ideal sheaf on $\rho\D^n$, being locally a finite intersection of the ideal sheaves of denominators, and all of its stalks are nonzero.
	By Cartan's Theorem~A, it has a global section $h\in\cO(\rho\D^n)$, $h\not\equiv0$, which is a common denominator of $\tilde f_1,\dots,\tilde f_m$.
	Put $T=\{h=0\}$, taking $h=1$ if $m=0$; then $\Phi^{-1}(P)\cap H'\subset T$.
	The open set $G=\Phi^{-1}((\Omega\cap V)\setminus P)\cap\rho\D^n$ is pseudoconvex, since it is biholomorphic to the intersection of the Stein open sets $(\Omega\cap V)\setminus P$ and $\Phi(\rho\D^n)$ of $Y$, which is Stein.
	As $H'\setminus T\subset G$, Lemma~\ref{lemma:hypersurface}, applied after rescaling $\rho\D^n$ to $\D^n$, gives $\rho\D^n\setminus T\subset G\subset\Phi^{-1}(\Omega)$.
\end{proof}

For meromorphically pseudoconvex domains, Proposition~\ref{proposition:hull} yields pseudoconvexity at boundary points having a neighborhood $W$ such that, for every open set $N\subset W$, the complement $N\setminus\Omega$ either is a hypersurface of $N$ or is not contained in any hypersurface of $N$.

\begin{proposition}
	\label{proposition:local}
	Let $\Omega$ be a meromorphically pseudoconvex domain in a complex manifold $Y$, and let $y_0\in\partial\Omega$.
	Suppose that $y_0$ has a neighborhood $W$ such that, for every open set $N\subset W$, if $N\setminus\Omega$ is contained in a hypersurface of $N$, then $N\setminus\Omega$ is itself a hypersurface of $N$.
	Then $\Omega$ is pseudoconvex at $y_0$.
	This assumption holds if $\Omega$ coincides with the interior of its closure near $y_0$.
\end{proposition}

\begin{proof}[Proof of Proposition~\ref{proposition:local}]
	We may assume $n\geq2$.
	Let $V$ be an open neighborhood of $y_0$ in which $\Omega$ is meromorphically pseudoconvex, let $U\subset V\cap W$ be a coordinate ball centered at $y_0$, and suppose that $\Omega\cap U$ is not pseudoconvex.
	By the preceding Hartogs characterization, there is a Hartogs figure $\Phi$ with $\Phi(\overline{\D^n})\subset U$, $\Phi(H)\subset\Omega\cap U$ and $\Phi(\D^n)\not\subset\Omega$.
	Choose $\max\{a,b\}<\rho<1$ so close to $1$ that $\Phi(\rho\D^n)\not\subset\Omega$, and let $T$ be the hypersurface of $\rho\D^n$ given by Proposition~\ref{proposition:hull}.
	Then $N=\Phi(\rho\D^n)\subset W$ is an open set such that $N\setminus\Omega$ is nonempty and contained in the hypersurface $\Phi(T)$ of $N$.
	By assumption, $S=\Phi^{-1}(N\setminus\Omega)$ is a nonempty hypersurface of $\rho\D^n$, and it does not meet $H(a,b)$ since $\Phi(H)\subset\Omega$.
	The complement $\rho\D^n\setminus S$ is pseudoconvex and contains the image $H(a,b)\cap\rho\D^n$ of a Hartogs figure with hull $\rho\D^n$, so $S=\emptyset$ by \cite[Theorem~II.3.4]{Fritzsche2002}, a contradiction.
	Thus $\Omega\cap U$ is pseudoconvex, hence Stein.

	Suppose next that $\Omega\cap W=\operatorname{int}\overline\Omega\cap W$, and let $N\subset W$ be an open set such that $N\setminus\Omega$ is contained in a hypersurface $Z$ of $N$.
	Then $N\setminus Z\subset\Omega$ is dense in $N$, so $N\subset\overline\Omega$, and hence $N\subset\operatorname{int}\overline\Omega\cap W=\Omega\cap W$.
	Thus $N\setminus\Omega=\emptyset$, and the assumption holds.
\end{proof}

\begin{proof}[Proof of Corollary~\ref{corollary:equivalence}]
	(1)$\Rightarrow$(2): If $\Omega$ is Gromov elliptic, then it is Oka by Gromov's theorem \cite{Gromov1989JAMS}, and it is meromorphically pseudoconvex by Theorem~\ref{theorem:main}, hence pseudoconvex by Proposition~\ref{proposition:local}, hence Stein by the theorem of Docquier and Grauert \cite{Docquier1960MA}.

	(2)$\Rightarrow$(1): Every Stein Oka manifold is Gromov elliptic \cite{Gromov1989JAMS}, \cite[Proposition~5.6.15]{Forstneric2017}.

	(2)$\Rightarrow$(3): Every Stein manifold admits a complete K\"ahler metric \cite{Grauert1956MA}.

	(3)$\Rightarrow$(2): Since $Y$ is Stein and $\Omega=\operatorname{int}\overline\Omega$ admits a complete K\"ahler metric, $\Omega$ is Stein by \cite[Theorem~1(b)]{Diederich1981MA}.
\end{proof}

\section{Proof of Corollary~\ref{corollary:applications}}
\label{section:applications}

\begin{proof}[Proof of (1)]
	Since $K=\overline{\operatorname{int}K}$, the domain $\Omega=\C^n\setminus K$ satisfies $\operatorname{int}\overline\Omega=\C^n\setminus\overline{\operatorname{int}K}=\Omega$.
	If $\Omega$ were Gromov elliptic, it would be Stein by Corollary~\ref{corollary:equivalence}.
	However, since $K\neq\emptyset$ is compact and $n\geq2$, every holomorphic function on $\Omega$ extends to $\C^n$ by Hartogs's extension theorem, so $\Omega$ is not Stein.
	Since $\overline{\B^n}$ is holomorphically convex, $\C^n\setminus\overline{\B^n}$ is an Oka manifold \cite{Kusakabe2024AoM}.
\end{proof}

\begin{proof}[Proof of (2)]
	A holomorphic function vanishing on a nonempty relatively open subset of $\mathbb R^n$ vanishes identically near it.
	Hence, if $N\subset\C^n$ is an open set such that $N\cap\mathbb R^n$ is contained in a hypersurface of $N$, then $N\cap\mathbb R^n=\emptyset$.
	Thus the domain $\Omega=\C^n\setminus\mathbb R^n$ satisfies the assumption of Proposition~\ref{proposition:local} at every boundary point.
	If $\Omega$ were Gromov elliptic, it would be meromorphically pseudoconvex by Theorem~\ref{theorem:main}, hence pseudoconvex by Proposition~\ref{proposition:local}, and hence Stein by the theorem of Docquier and Grauert \cite{Docquier1960MA}.
	In particular, the function $u(z)=-\log\dist(z,\mathbb R^n)=-\log|\operatorname{Im}z|$ would be plurisubharmonic on $\Omega$.
	However,
	\[
		\frac{\partial^2u}{\partial z_2\partial\overline z_2}(i,0,\dots,0)=-\frac14,
	\]
	a contradiction.
	The Oka property of $\Omega$ was recalled after Corollary~\ref{corollary:applications}.
\end{proof}

\begin{proof}[Proof of (3)]
	Let $\Omega\subsetneq\Pj^n$ be a Gromov elliptic domain with Lipschitz boundary.
	Then $\Omega=\operatorname{int}\overline\Omega$, and $\Omega$ is meromorphically pseudoconvex by Theorem~\ref{theorem:main}, so $\Omega$ is pseudoconvex by Proposition~\ref{proposition:local}.
	By the theorem of Ohsawa and Sibony \cite{Ohsawa1998NMJ} for $\mathcal C^2$ boundaries and its extension to Lipschitz boundaries by Harrington \cite[Theorem~1.1]{Harrington2017JGA}, both based on Takeuchi's theorem \cite{Takeuchi1964JMSJ}, $\Omega$ admits a bounded plurisubharmonic exhaustion function, that is, a plurisubharmonic function $\varphi<0$ on $\Omega$ such that $\{\varphi<c\}\Subset\Omega$ for every $c<0$; it is not constant since $\Omega$ is not compact.
	This is impossible, since $\Omega$ is Oka \cite{Gromov1989JAMS} and Oka manifolds carry no nonconstant plurisubharmonic functions bounded above \cite[Proposition~7.1.9]{Forstneric2017}.
\end{proof}

\section*{Acknowledgments}
This work was supported by the Japan Society for the Promotion of Science (JSPS) KAKENHI Grant Numbers JP21K20324 and JP24K16919.

\section*{AI disclosure}
Generative AI tools were used to assist with the preparation and revision of this manuscript.
The author takes full responsibility for its content.

\raggedbottom

\end{document}